\documentclass[twoside,centertags                  %,draft
]{amsart}

\usepackage[cmtip,color,all,curve,
arc,poly
]{xy}

\usepackage{amssymb}
\usepackage{graphicx}
\usepackage{mathrsfs}

\usepackage{float}

\newtheorem{thm}[equation]{Theorem}
\newtheorem{cor}[equation]{Corollary}
\newtheorem{lem}[equation]{Lemma}
\newtheorem{prop}[equation]{Proposition}
\theoremstyle{definition}

\theoremstyle{remark}
\newtheorem{rem}[equation]{Remark}
\newtheorem{exm}[equation]{Example}

\newcommand{\C}[1]{\mathscr{#1}}

\newcommand{\ob}{\operatorname{Ob}}

\newcommand{\cats}[2]{\operatorname{Cat}_{#2}(#1)}

\def\r{\rightarrow} % flecha -->
\def\into{\rightarrowtail}
\def\onto{\twoheadrightarrow}

\newcommand{\id}[1]{\mathrm{id}_{#1}}

\def\hom{\operatorname{Hom}}

\def\ho{\operatorname{Ho}}

\def\st{\stackrel} % abreviatura de \stackrel

\def\unit{\mathbf{1}} % abreviatura de \underline

\def\To{\longrightarrow}

\newcounter{rpage}
\newcounter{trunco}
\begin{document}

\title%[Strictifying homotopy algebras]
{On the unit of a monoidal model category}%
\author{Fernando Muro}%
\address{Universidad de Sevilla,
Facultad de Matem\'aticas,
Departamento de \'Algebra,
Avda. Reina Mercedes s/n,
41012 Sevilla, Spain}
\email{fmuro@us.es}
\urladdr{http://personal.us.es/fmuro}


\begin{abstract}
In this paper we show how to modify cofibrations in a monoidal model category so that the tensor unit becomes cofibrant while keeping the same weak equivalences. We obtain aplications to enriched categories and coloured operads in stable homotopy theory.
\end{abstract}

\maketitle

% ---------------------------------------------------------------------------------

%\numberwithin{equation}{section}

A \emph{monoidal model category} is a model category $\C M$ with a monoidal structure, consisting of a tensor product $\otimes\colon\C M\times\C M\r\C M$, a unit $\unit\in\ob\C M$, and coherent associativity and unit isomorphisms, such that the following two axioms hold:
\begin{itemize}
\item \emph{Push-out product axiom}: Given cofibrations $f\colon X\r Y$ and $g\colon U\r V$, their push-out product $f\odot g\colon X\otimes V\cup_{X\otimes U}Y\otimes U\r Y\otimes V$ is a cofibration. Moreover, if $f$ or $g$ is a trivial cofibration then so is $f\odot g$.
\item \emph{Unit axiom}: There exists a cofibrant resolution of the tensor unit $q\colon\tilde\unit \st{\sim}\r\unit$ (i.e.~a weak equivalence with cofibrant source) such that, for any cofibrant object $X$ in $\C M$, $q\otimes X$ and $X\otimes q$ are weak equivalences.
\end{itemize}
This is essentially Hovey's definition \cite[\S4]{hmc} with Schwede--Shipley's terminology \cite{ammmc}. It induces a monoidal structure on the homotopy category $\ho\C M$. In the symmetric case, if we want to equip the category of monoids with a transferred model structure, we can include the \emph{monoid axiom} \cite[Definition 3.3]{ammmc}. 

In recent applications, there seems to be a pressing need for a cofibrant tensor unit $\unit$, e.g.~\cite{htec, lurieha, amseco}. However, examples with non-cofibrant tensor units, such as $S$-modules \cite{ekmm} or symmetric and diagram spectra with the positive stable model structure \cite{acmccrs, mcds}, are indispensable in brave new algebraic geometry \cite[\S2.4]{hagII}.  Lewis--Mandell \cite{mmmc} and more recently the author \cite{htnso2, manso, dkhtec} developed some techniques to deal with non-cofibrant tensor units under mild extra assumptions. One of them is the \emph{very strong unit axiom}, which is the strengthening of the unit axiom where $X$ can be any object. This new axiom holds in all monoidal model categories known to the author since, in all of them, tensoring with a cofibrant object preserves weak equivalences, see Corollary \ref{bueno} below. %In particular, the very strong unit axiom holds in monoidal model categories with cofibrant tensor unit (taking $q$ to be the identity in $\unit$), in symmetric spectra with 
% the positive stable model structure, and 
%more generally 
%in the base category of any homotopical algebraic geometry context in the sense of \cite{hagII}.

In this paper, we prove that we can equip any suitable monoidal model category with a different model structure with the same weak equivalences where the tensor unit is cofibrant. This new model structure is minimal in a certain sense. 

Any monoidal category has an underlying set functor $\C M(\unit,-)\colon\C M\r\operatorname{Set}$. A map in $\C M$ is said to be \emph{surjective} if the induced map on underlying sets is surjective. Notice that the tensor unit is cofibrant in $\C M$ if and only if all trivial fibrations are surjective.

\begin{thm}\label{1}
Any combinatorial monoidal model category $\C M$ satisfying the very strong unit axiom admits a combinatorial monoidal model structure $\tilde{\C M}$ with the same weak equivalences and whose trivial fibrations are the surjective trivial fibrations in $\C M$. If $\C M$ is right or left proper then so is $\tilde{\C M}$. If $\C M$ is symmetric and satisfies the monoid axiom then so does $\tilde{\C M}$.
\end{thm}

\begin{exm}\label{ex}
Let $\C M=\operatorname{Sp}^\Sigma$ be the category of symmetric spectra of simplicial sets equipped with the positive stable model structure \cite[Proposition 3.1]{acmccrs}, where the sphere spectrum $\unit=S$ is not cofibrant. It is proper, symmetric,  and satisfies the monoid axiom. 
The very strong unit axiom is a consequence of Corollary \ref{bueno}, \cite[Lemma 5.4.4]{se}, and the fact that cofibrations in the positive stable model structure are also cofibrations in the ordinary stable model structure. Theorem \ref{1} applies and trivial fibrations in $\tilde{\C M}$ are the maps $f\colon X\r Y$ such that $f_{n}\colon X_{n}\r Y_{n}$ is a trivial Kan fibration for any $n>0$ and $f_{0}\colon X_{0}\r Y_{0}$ is surjective on vertices. 

The model structure $\tilde{\C M}$ is, strictly, between the ordinary and the positive stable model structures. Indeed, $\tilde{\C M}\neq\C M$ since $\unit$ is cofibrant in the former but not in the latter. We now exhibit a trivial fibration in $\tilde{\C M}$ which is not an ordinary stable trivial fibration. Let $X$ be a fibrant replacement of the sphere spectrum in the ordinary stable model structure and let $X'\subset X$ be the subspectrum with $X_n'=X_n$ for $n>0$ and $X_0'=$ the discrete simplicial set with the same vertices as $X_0$. The Kan complex $X_0$ is not discrete since its homotopy groups are the stable homotopy groups of the sphere spectrum, therefore $X_0'\subset X_0$ is not a trivial Kan fibration. In particular, $X'\subset X$ is a trivial fibration in $\tilde{\C M}$ which is not an ordinary stable trivial fibration. 

As far as we know, the model structure $\tilde{\C M}$ on symmetric spectra is new.
\end{exm}

We will actually prove the following result, with weaker but uglier hypotheses. Denote by $\varnothing$ the initial object of $\C M$.

\begin{thm}\label{2}
Let $\C M$ be a cofibrantly generated monoidal model category satisfying the very strong unit axiom for a certain cofibrant resolution $q\colon\tilde\unit\st{\sim}\r \unit$. Let
$$\tilde\unit\amalg\unit\st{j}\into C\mathop{\r}^p_\sim \unit$$
be a factorization of $(q,\id{\unit})\colon \tilde\unit\amalg\unit\r \unit$ into a cofibration followed by a weak equivalence in $\C M$ and let $i_1\colon \tilde\unit\r\tilde\unit\amalg\unit$ be the inclusion of the first factor of the coproduct. 
Assume that $\C M$ has %presentable tensor unit $\unit$ and 
sets $I$ and $J$ of generating cofibrations and generating trivial cofibrations, respectively, such that 
the domains of $I$ are small relative to $\tilde I$-cell for $\tilde I=I\cup\{\varnothing\r \unit\}$ and $\tilde\unit$ and the domains of $J$ are small relative to $\tilde J$-cell for $\tilde J=J\cup\{ji_1\colon \tilde \unit\r C\}$. 
Then there is a cofibrantly generated monoidal model category $\tilde{\C M}$ with the same underlying category and weak equivalences as $\C M$, set of generating cofibrations $\tilde I$ and set of generating trivial cofibrations $\tilde J$. If $\C M$ is right or left proper then so is $\tilde{\C M}$.  If $\C M$ is symmetric and satisfies the monoid axiom then so does $\tilde{\C M}$.
\end{thm}

Notice that the identity functor is a  monoidal Quillen equivalence $\C M\rightleftarrows\tilde{\C M}$ in the sense of \cite[Definition 4.2.16]{hmc}. 

\begin{rem}
Generating cofibrations $\tilde I$ in $\tilde{\C M}$ do not depend on any choice, hence the whole model structure is independent of choices. The factorization of $(q,\id{\unit})$ can be constructed by taking a cylinder $\bar\jmath =(\bar\jmath_1,\bar\jmath_2)\colon \tilde\unit\amalg\tilde\unit\into\tilde C$ for $\tilde \unit$ in $\C M$ and then the push-out of $\bar \jmath_2$ along $q$. Therefore, $X$ is fibrant in $\tilde{\C M}$ if and only if it is fibrant in $\C M$ and any map $\tilde\unit\r X$ is homotopic in $\C M$ to a map which factors through $q\colon\tilde \unit\st{\sim}\r \unit$. This property holds for a certain cofibrant replacement of the tensor unit if and only if it holds for anyone. 
\end{rem}

Theorem \ref{2} is better suited for model categories of topological nature, as in the following two examples.

\begin{exm}\label{ex2}
Let $\C M$ be any of the symmetric monoidal model categories of diagram spectra built upon the category $\operatorname{Top}_*$ of pointed compactly generated toplogical spaces \cite[Definition 2.4.21 (3)]{hmc} with the positive stable model structure in \cite[Theorem 14.2]{mcds}, i.e.~symmetric spectra $\Sigma\C S$, orthogonal spectra $\C I\C S$, or the category of $\C W$-spaces $\C W\C T$. They are proper and satisfy the monoid axiom. The very strong unit axiom follows from Corollary \ref{bueno}, \cite[Proposition 12.7]{mcds}, and the fact that cofibrations in the positive stable model structure are also cofibrations in the ordinary stable model structure \cite[Theorem 9.2]{mcds}. 

The category $\C M$ is (co)tensored over $\operatorname{Top}_*$. Homotopies are maps from the cylinders constructed by smashing with the interval with an outer base point $[0,1]_+$. An \emph{$h$-cofibration} is a map satisfying the homotopy extension property. These maps can be characterized by the left lifting property with respect to a class of maps, compare the proof of \cite[Theorem 3.1]{smss}, hence $h$-cofibrations are closed under retracts, push-outs and transfinite compositions. Ordinary (and hence positive) stable cofibrations in $\C M$ are $h$-cofibrations \cite[Lemma 5.5, Definition 5.9, Theorem 9.2]{mcds}. The $h$-cofibrations in $\operatorname{Top}_*$ are closed inclusions, in particular $h$-cofibrations in $\C M$  are spacewise closed inclusions.

%Inclusions of retracts and $h$-cofibrations in $\operatorname{Top}_*$ are closed inclusions, in particular inclusions of retracts and $h$-cofibrations in $\Sigma\C S$  are spacewise closed inclusions. Push-outs of $\varnothing\r\unit$ are inclusions $X\r X\amalg\unit$ of the first factor, which admit a retraction $(\id{X},0)\colon X\amalg\unit\r X$. A push-out of $ji_1$ is the composite of an $X\r X\amalg\unit$ and a positive cofibration, so it is a spacewise closed inclusion. 

All objects in $\operatorname{Top}_*$ are small relative to closed inclusions by cardinality reasons, therefore all objects in $\C M$ are small relative to $\tilde I$-cell and $\tilde J$-cell for any choice of $I$, $J$ and the factorization of $(q,\id{\unit})$, since maps in $\tilde I$ and $\tilde J$ are ordinary stable cofibrations in $\C M$.

% Let $I$ and $J$ be the sets of generating cofibrations and generating trivial cofibrations indicated in \cite[Theorem 14.2]{mcds}, i.e.~$F^+I$ and $K^+$, respectively. 
% Let us fix the cofibrant resolution of the sphere spectrum in \cite[Definition 8.4 and Lemma 8.6]{mcds}, $q=\lambda_0\colon\tilde\unit =F_1S^1\st{\sim}\rightarrow F_0S^0=\unit$. Take any factorization of $(q,\id{\unit})$. Any relative $\tilde I$- or $\tilde J$-cell complex is a transfinite composition of cofibrations in $\C M$ and inclusions of first factors $X\r X\amalg\unit$. Cofibrations in $\C M$ are $h$-cofibrations in the sense of \cite[\S5]{mcds}, i.e.~they satisfy the homotopy extension property. Here we use \cite[Cofibration Hypothesis 5.3 and Lemma 5.5]{mcds}. The map $X\r X\amalg\unit$ is also an $h$-cofibration for any symmetric spectrum $X$ for obvious reasons (any homotopy from $X$ starting at the restriction of a map from $X\amalg \unit$ extends to a homotopy from $X\amalg \unit$ by considering the trivial homotopy on the second factor). Therefore, \cite[Definition 5.6 and Lemma 5.7]{mcds} prove the smallness conditions for $\tilde\unit =F_1S^1$ and for the sources of $I$ and $J$.

A trivial fibration in $\tilde{\C M}$ is a map $f\colon X\r Y$ such that $f_{n}\colon X_{n}\r Y_{n}$ is a Serre fibration and a weak equivalence of spaces for $n>0$ and $f_{0}\colon X_{0}\r Y_{0}$ is surjective. 
Taking the cofibrant resolution of the sphere spectrum in \cite[Definition 8.4 and Lemma 8.6]{mcds}, $q=\lambda_0\colon\tilde\unit =F_1S^1\st{\sim}\rightarrow F_0S^0=\unit$, we see that an object $X$ is fibrant in $\tilde{\C M}$ if and only if it is a positive $\Omega$-spectrum such that the structure map $X_0\r\Omega X_1$ induces a surjection on $\pi_0$. This characterization of fibrant objects is also valid in Example \ref{ex}.

It is possible to check, as in Example \ref{ex}, that $\tilde{\C M}$ is strictly between $\C M$ and the ordinary stable model structure, e.g.~in symmetric or orthogonal spectra, if $X$ is an ordinary stable fibrant replacement of the sphere spectrum and $X'$ is defined as $X'_{n}=X_{n}$, $n>0$, and $X_{0}'=$ the set $X_{0}$ with the discrete topology, the identity on underlying sets induces a map $X'\r X$ which is a trivial fibration in $\tilde{\C M}$ but not an ordinary stable trivial fibration.

We have not previously seen the model structure $\tilde{\C M}$ in the literature.
\end{exm}

\begin{exm}\label{ex3}
Let $\C M=\C M_S$ be the model category of $S$-modules \cite[Theorem VII.4.6]{ekmm}. It is a cofibratly generated symmetric monoidal model category satisfying the monoid axiom, see Proposition \ref{sar} below, and it is right proper since all objects are fibrant. The very strong unit axiom follows from Corollary \ref{bueno} and \cite[Theorem III.3.8]{ekmm}. 

Cofibrations in $\C M_S$ are spacewise closed inclusions, see \cite[Cofibration hypothesis and the paragraph afterwards]{ekmm} and \cite[App.~Proposition 3.9]{esht}. Inclusions of retracts in $\operatorname{Top}_*$ are closed, hence inclusions of retracts in $\C M_S$ are spacewise closed inclusions.

A push-out of $\varnothing\r\unit$ is the same as an inclusion of first factor $X\r X\amalg\unit$. This map admits a retraction $(\id{X},0)\colon X\amalg\unit\r X$. A push-out of $ji_1$ is a composite of such an inclusion $X\r X\amalg\unit$ and a cofibration in $\C M_S$, so it is a spacewise closed inclusion.

The smallness condition follows for any choice of $I$, $J$ and factorization of $(q,\id{\unit})$, since all objects in $\operatorname{Top}_*$ are small relative to closed inclusions. 

% 
% 
% The sets of generating cofibrations and generating trivial cofibrations in $\C M$ are $I=\{f_{q,n}\colon S\wedge_{\C L}\mathbb L\Sigma^\infty_q S^n\r S\wedge_{\C L}\mathbb L\Sigma^\infty_q CS^n\}_{q,n\geq 0}$ and $J=\{g_{q,n}\colon S\wedge_{\C L}\mathbb L\Sigma^\infty_q (CS^n\wedge \{0\}_+)\r S\wedge_{\C L}\mathbb L\Sigma^\infty_q (CS^n\wedge [0,1]_+)\}_{q,n\geq 0}$, respectively, see \cite[Theorem VII.4.14 and the proof of Lemma VII.5.6]{ekmm}. Consider the cofibrant resolution of the sphere spectrum $q\colon \tilde \unit=S\wedge_{\C L}\mathbb{L}S\st{\sim}\r S=\unit$ in \cite[\S II.1]{ekmm}, and any factorization of $(q,\id{\unit})$. Any relative $\tilde I$- or $\tilde J$-cell complex is a transfinite composition of cofibrations in $\C M$ and inclusions of first factors $X\r X\amalg\unit$. Cofibrations in $\C M$ are spacewise closed inclusions (compare the proof of \cite[Lemma VII.5.2]{mcds}). The map $X\r X\amalg\unit$ is also as spacewise closed 
% inclusion since it 
% admits a retraction, e.g.~$(\id{X},0)\colon X\amalg\unit\r X$. Therefore, \cite[Proposition III.1.7]{ekmm} (which should demand inclusions to be closed, compare \cite[Proposition 2.4.2]{hmc}) proves the smallness condition for $\tilde\unit =S\wedge_{\C L}\mathbb{L}S$  and for the sources of $I$ and $J$.

Taking the cofibrant resolution of the sphere spectrum $q\colon \tilde \unit=S\wedge_{\C L}\mathbb{L}S\st{\sim}\r S=\unit$ in \cite[\S II.1]{ekmm}, we see that an $S$-module $X$ is fibrant in $\tilde{\C M}_S$ if and only if any map of spectra $S\r X$ is homotopic to a map of $S$-modules.
 
Such a modification of the model category of $S$-modules turning the sphere spectrum into a cofibrant object seems to be new in the literature.
\end{exm}

% \begin{rem}
% A model structures on the category of small categories enriched in symmetric spectra has been considered in \cite{htsc}. However, to the best of our knowledge, there is no similar results for $S$-modules. Berger--Moerdijk \cite[Theorem 1.10]{htec} recently proved a general result about the existence of nice induced model structures on $\C M$-enriched small categories for $\C M$ an appropriate symmetric monoidal model category. One of the hypotheses of this theorem is the cofibrancy of the tensor unit. This rules out the standard model structure $\C M=\C M_S$ on $S$-modules, but not the new one $\tilde{\C M}$. We now check that $\tilde{\C M}$ does satisfy all the hypotheses.
% 
% A $\otimes$-cofibration in $\tilde{\C M}$ is a relative $(\tilde I\otimes\ob(\C M))$-cell complex. A map in $I\otimes \ob(\C M)$ is spacewise of the form $\Sigma^nX\r C\Sigma^nX$, $n\geq 0$, which is a closed inclusion. Moreover, $(\varnothing\r\unit)\otimes X=(\varnothing\r X)$ is spacewise a closed inclusion since points are closed in compactly generated spaces. Hence $\otimes$-cofibrations in $\tilde{\C M}$ are spacewise closed inclusions. This implies that all $S$-modules are $\otimes$-small and that the sources of maps in $\tilde I$ are $\otimes$-finite. Hence it follows that $\tilde {\C M}$ is compactly generated, see the third paragraph after \cite[Definition 1.2]{htec}. Adequacy follows fro \cite[Lemma 1.3]{htec} since all objects in 
% 
% 
% \end{rem}

We start with a clarification concerning the unit axiom.

\begin{lem}
We can replace `there exists a' with `for any' in the definition of the unit axiom.
\end{lem}

\begin{proof}
Let $q'\colon \tilde\unit'\st{\sim}\onto\unit$ be a fixed cofibrant resolution of the tensor unit which is a trivial fibration. It suffices to prove that the monoid axiom is satisfied for some $q\colon \tilde\unit\st{\sim}\r\unit$ if and only if it is satisfied for $q'$. Since $\tilde\unit$ is cofibrant and $q'$ is a trivial fibration, we can factor $q$ as $q=q'f$ for a certain $f\colon\tilde\unit\st{\sim}\r\tilde\unit'$. This map $f$ is a weak equivalence by the 2-out-of-3 axiom. By the push-out product axiom and Ken Brown's lemma \cite[Lemma 1.1.12]{hmc}, tensoring with a cofibrant object $X$ preserves weak equivalences between cofibrant objects, so $f\otimes X$ and $X\otimes f$ are weak equivalences. The 2-out-of-3 axiom applied to $q\otimes X=(q'\otimes X)(f\otimes X)$ and $(X\otimes q)=(X\otimes q')(X\otimes f)$ proves the claim.
\end{proof}

The analogous result for the very strong unit axiom need not hold in general. It does hold, with essentially the same proof, if $\C M$ is symmetric and satisfies the monoid axiom.

The following characterization of the very strong unit axiom is used in the proof of Theorem \ref{2}. It is essentially \cite[Lemmas A.4 and A.5]{manso}. We offer here a full proof to clear any doubt about the necessity of the monoid axiom, which is always assumed therein. Part of this proof is due to David White \cite{whimo}.

\begin{lem}\label{white}
Let $\C M$ be a monoidal model category. Then $(1)\Leftrightarrow(4)+(5)$, $(2)\Leftrightarrow (4)$, and $(3)\Leftrightarrow (5)$, where:
\begin{enumerate}
 \item $\C M$ satisfies the very strong unit axiom for a cofibrant resolution $q\colon\tilde\unit\st{\sim}\r\unit$.
 \item The functor $\tilde\unit\otimes-\colon\C M\r\C M$ preserves weak equivalences.
 \item The functor $-\otimes\tilde\unit\colon\C M\r\C M$ preserves weak equivalences.
 \item The functor $\tilde\unit\otimes-\colon\C M\r\C M$ preserves and reflects weak equivalences.
 \item The functor $-\otimes\tilde\unit\colon\C M\r\C M$ preserves and reflects  weak equivalences.
\end{enumerate}
\end{lem}

\begin{proof}
Clearly, $(4)\Rightarrow(2)$ and $(5)\Rightarrow(3)$. If $q\otimes X$ is a weak equivalence for any $X$, $(4)$ follows by applying the the 2-out-of-3 axiom to the following commutative square, 
$$\xymatrix{
\tilde \unit\otimes X\ar[d]_{q\otimes X}^\sim\ar[r]^{\tilde \unit\otimes f}&\tilde \unit\otimes Y\ar[d]^{q\otimes Y}_\sim\\
X\ar[r]^f&Y
}$$
Tensoring in the reverse order, we see that if $X\otimes q$ is a weak equivalence for any $X$ then $(5)$ holds. In particular $(1)\Rightarrow (4)+(5)$.

Assuming $(2)$, and given an object $X$ in $\C M$ with a cofibrant resolution $q'\colon\tilde X\st{\sim}\r X$, $q\otimes X$ is a weak equivalence by  the 2-out-of-3 axiom applied to the following commutative diagram
$$\xymatrix{
\tilde \unit\otimes \tilde X\ar[d]_{q\otimes \tilde X}^\sim\ar[r]_\sim^{\tilde \unit\otimes q'}&\tilde \unit\otimes X\ar[d]^{q\otimes X}\\
\tilde X\ar[r]^{q'}_\sim&X
}$$
Here $q\otimes\tilde X$ is a weak equivalence by the unit axiom and $\tilde\unit\otimes q'$ is a weak equivalence by (2). Tensoring in the reverse order, we check that $(3)$ implies that $X\otimes q$ is a weak equivalence. This completes the proof.
\end{proof}

\begin{cor}\label{bueno}
If tensoring with a cofibrant object, from the left or from the right, preserves weak equivalences in $\C M$ then the very strong unit axiom holds in $\C M$ for any cofibrant resolution of the tensor unit.
\end{cor}

The following lemma is needed in order to check the left properness statement.

\begin{lem}\label{10}
With the notation in Theorem \ref{2}, a relative $\tilde I$-cell complex $X\r Y$ is the same as a composite $X\r X\amalg \unit^{(S)}\r Y$
where $\unit^{(S)}$ is a coproduct of copies of $\unit$ indexed by a set $S$, the first arrow is the inclusion of the first factor, and the second arrow is a relative $I$-cell complex
\end{lem}

This follows from the fact that, in the construction of a relative $\tilde I$-cell complex, we can move all occurrences of $\varnothing\r\unit$ to the beginning.

We can now tackle the proof of Theorem \ref{2}.

\begin{proof}[Proof of Theorem \ref{2}]
We use the characterization of cofibrantly generated model categories in \cite[Theorem 2.1.19]{hmc}. We must check that $\tilde{\C M}$ satisfies six conditions which are satisfied by $\C M$. Condition 1, about weak equivalences, holds since $\tilde{\C M}$ has the same weak equivalences as $\C M$.  Conditions 2 and 3 are part of the assumptions. 

Let us check that relative $\tilde J$-cell complexes are weak equivalences as well as $\tilde I$-cofibrations (4). Maps in $J$ are $I$-cofibrations, $i_1\colon \tilde\unit\r\tilde\unit\amalg\unit$ is a push-out of $\varnothing\r \unit$, which is in $\tilde I$, and  $j\colon\tilde\unit\amalg\unit\r C$ is an $I$-cofibration. Hence all maps in $\tilde J$, and more generally all relative $\tilde J$-cell complexes, are $\tilde I$-cofibrations. We must also show that any relative $\tilde J$-cell complex $f\colon X\r Y$ is a weak equivalence. By the very strong unit axiom and Lemma \ref{white}, it is enough to show that $\tilde\unit \otimes f$ is a weak equivalence. Notice that $\tilde\unit \otimes f$ is a relative $(\tilde \unit\otimes \tilde J)$-cell complex. Hence it is enough to prove that maps in $\tilde \unit\otimes \tilde J=\tilde \unit\otimes J\cup\{\tilde \unit\otimes ji_1\}$ are $J$-cofibrations, or equivalently weak equivalences and $I$-cofibrations. The functor $\tilde\unit\otimes-\colon\C M\r\C M$  
preserves 
$I$- and $J$-cofibrations by 
the push-out product axiom, since $\tilde \unit$ is cofibrant. Hence $\tilde \unit\otimes J$ consists of $J$-cofibrations. Moreover, for the same reason $\tilde\unit\otimes j$ is an $I$-cofibration. The map $\tilde \unit\otimes i_1\colon \tilde \unit\otimes\tilde\unit\r\tilde\unit\otimes\tilde \unit\amalg\tilde \unit\otimes\unit$ is also an $I$-cofibration since $\tilde \unit\otimes\unit\cong \tilde \unit$ is cofibrant in $\C M$. Finally, since $ji_1$ is a weak equivalence, $\tilde\unit\otimes (ji_1)$ too, by the very strong unit axiom and Lemma \ref{white}.

Let us check that $\tilde I$-injective maps are $\tilde J$-injective weak equivalences (5). Any $\tilde I$-injective map is also $I$-injective, since $I\subset\tilde I$, so it is a $J$-injective weak equivalence. It remains to show that any $\tilde I$-injective map $f\colon X\r Y$ satisfies the right lifting property with respect to $ji_1$, i.e.~that we can find a lifting for any solid commutative square as follows
$$\xymatrix@C=40pt@R=10pt{
\tilde\unit\ar[r]^g\ar[d]_{i_1}&X\ar[dd]^f\\
\tilde\unit\amalg\unit\ar[d]_j\ar@{-->}[ru]^<(.2){(g,h')}&\\
C\ar[r]_h\ar@{-->}[ruu]_l&Y
}$$
Let $i_2\colon\unit\r\tilde\unit\amalg\unit$ be the inclusion of the second factor. Since $f$ is $\tilde I$-injective we can lift $hji_2\colon\unit \r Y$ along $f$. Denote a lifting by $h'\colon\unit\r X$. The upper dashed arrow in the previous diagram subdivides it into a commutative triangle (above) and a commutative square (below). This commutative square has a lifting $l$ since $j$ is an $I$-cofibration  $f$ is $I$-injective. This map $l$ is also a lifting of the solid diagram.

Let us prove that $\tilde J$-injective weak equivalences are $\tilde I$-injective (6). Any  $\tilde J$-injective weak equivalence $f\colon X\r Y$ is $J$-injective, and hence $I$-injective. We must prove that $f$ satisfies the right lifting property with respect to $\varnothing\r \unit$, i.e.~that we can find a lifting for any solid square as follows,
$$\xymatrix@C=50pt@R=10pt{
\varnothing\ar[r]^g\ar[d]&X\ar[ddd]^f\\
\tilde\unit\ar[d]_{ji_1}\ar@{-->}[ru]^<(.2){h'}&\\
C\ar[d]_p\ar@{-->}[ruu]^<(.2){h''}\\
\unit\ar[r]_h\ar@{-->}[ruuu]_{l}&Y
}$$
Since $\tilde \unit$ is cofibrant in $\C M$, we can find a dashed arrow $h'\colon\tilde\unit\r X$ subdividing the diagram in two commutative parts. Moreover, since $ji_1\in \tilde J$, there exists a map $h''\colon C\r X$ which further subdivides the bottom part of the diagram. Now the composite $$l\colon \unit\st{i_2}\To\tilde\unit\amalg\unit\st{j}\To C\st{h''}\To X$$
is the desired lifting.

In the previous paragraphs we have constructed the model structure $\tilde{\C M}$. We now check that it satisfies the push-out product axiom, so it is a monoidal model category with cofibrant tensor unit. More precisely, we must check that the push-out product $f\odot g$ is an $\tilde I$-cofibration if $f,g\in \tilde I$, or a $\tilde J$-cofibration if $f\in \tilde I$ and $g\in \tilde J$ or if $f\in \tilde J$ and $g\in \tilde I$. Since $\C M$ is a monoidal model category, we can skip the cases $f,g\in I$, $f\in I$  and $g\in J$, and $f\in J$ and $g\in I$.  If $f$ is $\varnothing\r \unit$ then $f\odot g=g$ and everything is trivial. Similarly if $g$ is $\varnothing\r \unit$. If $f=ji_1$ and $g\colon U\r V$ is in $I$, the push-out product $(ji_1)\odot g$ is the composite
$$\xymatrix{
\tilde\unit \otimes V\bigcup_{\tilde\unit \otimes U}C\otimes U\ar[rrr]^-{i_1\odot g\bigcup_{(\tilde\unit\amalg\unit)\otimes U}C\otimes U}&&&
(\tilde\unit\amalg\unit)\otimes V\bigcup_{(\tilde\unit\amalg\unit)\otimes U}C\otimes U\ar[r]^-{j\odot g}& C\otimes V}$$
and $i_1\odot g$ is 
$$i_1\odot g=\id{\tilde\unit\otimes V}\amalg\unit\otimes g\colon\tilde\unit\otimes V\amalg\unit\otimes U\To \tilde\unit\otimes V\amalg\unit\otimes V.$$
The map $i_1\odot g$ is a coproduct of $I$-cofibrations since $\unit\otimes g\cong g$, so $i_1\odot g\bigcup_{(\tilde\unit\amalg\unit)\otimes U}C\otimes U$ is an $I$-cofibration. Moreover,  $j\odot g$ is an $I$-cofibration by the push-out product axiom in $\C M$. Therefore $(ji_1)\odot g$ is an $I$-cofibration. We have already seen above that $\tilde \unit\otimes (ji_{1})$ is a $J$-cofibration. 
% The map $ji_{1}$ is also a weak equivalence, hence $\tilde \unit\otimes (ji_{1})$ too by Lemma \ref{white}. Moreover, $\tilde \unit\otimes j$ is an $I$-cofibration by the push-out product axion in $\C M$, and $\tilde \unit\otimes i_{1}\colon \tilde \unit\otimes \tilde \unit\r \tilde \unit\otimes \tilde \unit\amalg \tilde \unit\otimes\unit$ is an $I$-cofibration since $\tilde \unit\otimes  \unit\cong \tilde \unit$ is cofibrant in $\C M$. Hence $\tilde \unit\otimes (ji_{1})$ is an $I$-cofibration, and therefore a $J$-cofibration, since we have already seen that it is a weak equivalence. 
By the push-out product axiom in $\C M$, $(\tilde 
\unit\otimes (ji_{1}))\odot g=\tilde \unit\otimes ((ji_{1})\odot g)$ is also a $J$-cofibration, in particular a weak equivalence. Hence $(ji_{1})\odot g$ is a weak equivalence by Lemma \ref{white}, so it is  a $J$-cofibration, since we have already seen that it is an $I$-cofibration. In particular $(ji_{1})\odot g$ is a $\tilde J$-cofibration. If $g=ji_{1}$ and $f\in I$ the proof is similar.

The statement about right properness is obvious since $\tilde{\C M}$ has less fibrations than $\C M$. Suppose that $\C M$ is left proper. By Lemma \ref{10}, in order to check that $\tilde{\C M}$ is also left proper it is enough to prove that, for any weak equivalence $f$ and any set $S$, $f\amalg \unit^{(S)}$ is a weak equivalence. By Lemma \ref{white} it suffices to prove that $\tilde\unit\otimes f \amalg \tilde \unit^{(S)}$ is a weak equivalence, and this follows since $\tilde\unit\otimes f$ is a weak equivalence (again by Lemma \ref{white}), $\tilde \unit^{(S)}$ is cofibrant in $\C M$, and $\C M$ is left proper.

For the final part of the statement, we must check that any relative $(\tilde J\otimes\operatorname{Ob}\C C)$-cell complex $f$ is a weak equivalence. By Lemma \ref{white}, it suffices to show that $\tilde\unit\otimes f$ is a weak equivalence. The map $\tilde\unit\otimes f$ is a relative $(\tilde\unit\otimes \tilde J\otimes\operatorname{Ob}\C C)$-cell complex. We have seen above that $\tilde\unit\otimes \tilde J$ consists of $J$-cofibrations. Hence $\tilde\unit\otimes f$ is a relative $((\text{$J$-cofibrations})\otimes\operatorname{Ob}\C C)$-cell complex, so it is a weak equivalence by the monoid axiom in $\C M$.

%Similarly if $f\in J$ and $g=ji_1$. Finally, if $f=g=ji_2$ we can similarly decompose 
%\begin{align*}
%(ji_1)\odot(ji_1)&=(j\odot (ji_1))\left(i_1\odot(ji_1)\bigcup_{(\tilde\unit\amalg \unit)^{\otimes 2}}C\otimes(\tilde\unit\amalg \unit)\right),\\
%i_1\odot(ji_1)&=\id{\tilde\unit\otimes C}\amalg\unit\otimes ji_1\cong \id{\tilde\unit\otimes C}\amalg ji_1,\\
%j\odot (ji_1)&=(j\odot j)\left(
%(\tilde\unit\amalg\unit)\otimes C\bigcup_{(\tilde\unit\amalg \unit)^{\otimes 2}}j\odot i_1
%\right),\\
%j\odot i_1&= j\otimes\unit\amalg\id{C\otimes\tilde\unit}\cong j\amalg\id{C\otimes\tilde\unit}.
%\end{align*}
%Starting from the bottom, $j\odot i_1$ is a cofibration in $\C M$ since it is a coproduct of cofibrations in $\C M$, hence the second factor of $j\odot (ji_1)$ is a cofibration in $\C M$.  The first factor is $j\odot j$, which is a cofibration in $\C M$ by the push-out product axiom. We conclude that $j\odot (ji_1)$ is a a cofibration in $\C M$.
\end{proof}

We now characterize cofibrant objects in $\tilde{\C M}$. We say that an object in a monoidal model category is \emph{cofibrant mod $\unit$} if it is a retract of an object $X$ fitting in a cofibration $\unit^{(S)}\into X$. This terminology is justified because cofibrant objects are cofibrant mod $\unit$ (the set $S$ may be empty), and the converse holds if and only if $\unit$ is cofibrant. The $\unit$-cofibrant objects of \cite{htnso2} are cofibrant mod $\unit$. The following result is an immediate consequence of Lemma \ref{10}.

\begin{cor}
In the conditions of Theorem \ref{2}, an object is cofibrant mod $\unit$ in $\C M$ if and only if it is cofibrant in $\tilde{\C M}$.
\end{cor}

Let us consider the functorial properties of $\tilde{\C M}$. A Quillen adjunction between monoidal model categories $F\colon \C M\rightleftarrows\C N\colon G$ is \emph{weak monoidal} \cite[Definition 3.6]{emmc} if $F$ is colax monoidal, the comultiplication of $F$, 
$$F(X\otimes Y)\To F(X)\otimes F(Y),$$
is a weak equivalence when $X$ and $Y$ are cofibrant, and for some (and hence any) cofibrant resolution $q\colon\tilde\unit\st{\sim}\r\unit$ of the tensor unit in $\C M$, the composite
\begin{equation}\label{cou}
F(\tilde\unit)\st{F(q)}\To F(\unit)\To\unit,
\end{equation}
where the second map is the counit of $F$, is a weak equivalence. This generalizes Hovey's (strong) monoidal Quillen adjunctions \cite[Definition 4.2.16]{hmc}, where the comultiplication and the counit are required to be always isomorphisms, hence the remaining condition is that $F(q)$ be a weak equivalence.

\begin{prop}
Let $F\colon \C M\rightleftarrows\C N\colon G$ be a weak monoidal Quillen adjunction such that $\C M$ satisfies the assumptions of Theorem \ref{2}, $F(\unit)$ is cofibrant, and the counit $F(\unit)\r\unit$ is a weak equivalence. Then $F\colon \tilde{\C M}\rightleftarrows\C N\colon G$ is also a weak monoidal Quillen adjunction. The same holds if we replace `adjunction' with `equivalence'.
\end{prop}

\begin{proof}
By assumption, the maps in $F(I)$ and $F(J)$ are cofibrations and trivial cofibrations in $\C N$, respectively. The map $F(\varnothing\r\unit)=(\varnothing\r F(\unit))$ is assumed to be a cofibration. Hence $F\colon\tilde{\C M}\r\C N$ preserves cofibrations, in particular $F(ji_1)$ is a cofibration. Let us check that it is actually a trivial cofibration. The composite
$$F(\tilde\unit)\st{F(ji_1)}\To F(C)\st{F(p)}\To F(\unit)\st{\sim}\To \unit$$
is the weak equivalence \eqref{cou} since $pji_1=(q,\id{\unit})i_1=q\colon\tilde\unit\r\unit$. Hence, by the 2-out-of-3 axiom, $F(ji_1)$ is a weak equivalence if and only if $F(p)$ is a weak equivalence. In $\C M$, the inclusion of the second factor $i_2\colon\unit\r\tilde\unit\amalg\unit$ is a cofibration since $\tilde\unit$ is cofibrant, and moreover $ji_1$ is a trivial cofibration since $j$ is a cofibration, $p$ is a weak equivalence, and $pji_2=(q,\id{\unit})i_2=\id{\unit}$. Therefore $F(ji_2)$ is a trivial cofibration and $F(p)$ is a weak equivalence by the 2-out-of-3 axiom applied to $\id{\unit}=F(\id{\unit})=F(p)F(ji_2)$.
We conclude that $F\colon\tilde{\C M}\r\C N$  is a left Quillen functor.

We now check the weak monoidal part, i.e.~that the comultiplication is a weak equivalence when evaluated at objects $X$ and $Y$ which are cofibrant  mod $\unit$. Take cofibrant resolutions $q_X\colon \tilde X\st{\sim}\r X$ and $q_Y\colon \tilde Y\st{\sim}\r Y$ in $\C M$. These maps are weak equivalences between cofibrant objects in $\tilde{\C M}$. By the push-out product axiom in $\tilde{\C M}$, $q_X\otimes q_Y$ is also a weak equivalence between cofibrant objects in $\tilde{\C M}$. By Ken Brown's lemma, $F(q_X)$, $F(q_Y)$ and $F(q_X\otimes q_Y)$ are also weak equivalences between cofibrant objects in $\C N$. By the push-out product axiom in $\C N$, $F(q_X)\otimes F(q_Y)$ is a weak 
equivalence between cofibrant objects in $\C N$ too. Hence the comultiplication $F(X\otimes Y)\To F(X)\otimes F(Y)$ is a weak equivalence by the 2-out-of-3 axiom applied to the following commutative square, 
$$\xymatrix{
F(\tilde X\otimes\tilde Y)\ar[r]^-{\text{comult.}}_-\sim\ar[d]_{F(q_X\otimes q_Y)}^\sim
&
F(\tilde X)\otimes F(\tilde Y)\ar[d]^{F(q_X)\otimes F(q_Y)}_\sim\\
F(X\otimes Y)\ar[r]_-{\text{comult.}}
&
F(X)\otimes F(Y)
}$$

For the final statement, we have to check that, if $X$ is cofibrant mod $\unit$ in $\C M$ and $Y$ is fibrant in $\C N$,  $F(X)\r Y$ is a weak equivalence if and only if the adjoint map $X\r G(Y)$ is a weak equivalence. We are assuming that this is true if $X$  is cofibrant in $\C M$. Let $X$ be cofibrant mod $\unit$ and let $q_X\colon\tilde X\st{\sim}\r X$ be a cofibrant resolution in $\C M$. The adjoint of $$F(\tilde X)\mathop{\To}^{F(q_X)}_\sim F(X)\To Y$$
is $$\tilde X\mathop{\To}^{q_X}_\sim X\To G(Y).$$
We have seen above that $F(q_X)$ is a weak equivalence, hence the claim is a consequence of the 2-out-of-3 axiom.
\end{proof}

\begin{cor}
Let $F\colon \C M\rightleftarrows\C N\colon G$ be a weak monoidal Quillen adjunction such that $\C M$ and $\C N$ satisfy the assumptions of Theorem \ref{2}, $F(\unit)$ is cofibrant mod $\unit$, and the counit $F(\unit)\r\unit$ is a weak equivalence. Then $F\colon \tilde{\C M}\rightleftarrows\tilde{\C N}\colon G$ is also a weak monoidal Quillen adjunction. The same holds if we replace `adjunction' with `equivalence'.
\end{cor}

This corollary follows by composing $F\colon \C M\rightleftarrows\C N\colon G$ with the monoidal Quillen equivalence $\C N\rightleftarrows\tilde{\C N}$ defined by the identity functor. %The conditions on the tensor unit in the two previous results are automatically satisfied if $F$ is strong.

A map $f$ in a monoidal model category $\C M$ is a \emph{pseudo-cofibration} if $f\odot g$ and $g\odot f$ are (trivial) cofibrations whenever $g$ is a (trivial) cofibration, compare \cite[\S6]{dkhtec}. Cofibrations are examples of pseudo-cofibrations  and $\varnothing\r\unit$ too. If $\unit$ is cofibrant, pseudo-cofibrations are the same thing as cofibrations. An object $X$ in $\C M$ is \emph{pseudo-cofibrant} if $\varnothing\r X$  is a pseudo-cofibration. These objects were first considered by Lewis and Mandell \cite{mmmc} under the name of semicofibrant objects. They share many properties with cofibrant objects and have been very useful in \cite{htnso2, manso}.

Pseudo-cofibrations can be characterized as the maps satisfying the left lifting property with respect to a certain class of maps, compare the proof of \cite[Lemma 3.5]{ammmc}, hence they are closed under retracts, push-outs, and transfinite compositions. We deduce from Lemma \ref{10} that cofibrations in $\tilde{\C M}$ are pseudo-cofibrations in $\C M$. This inclusion may be strict, as we now see in examples.

\begin{prop}
If $\C M$ is any of the categories in Examples \ref{ex} and \ref{ex2}, ordinary stable cofibrations are pseudo\--co\-fib\-rations in the positive stable model structure.
\end{prop}

\begin{proof}
A positive stable cofibration $g\colon U\r V$ is the same as an ordinary stable cofibration such that $g_{0}\colon U_{0}\r V_{0}$ is an isomorphism, compare \cite[Theorem 14.1]{mcds}. If $f\colon X\r Y$ is an ordinary stable cofibration and $g\colon U\r V$ is a positive stable (trivial) cofibration, $f\odot g$ is an ordinary stable (trivial) cofibration by the push-out product axiom for the ordinary stable model structure \cite[Lemma 6.6 and Proposition 12.6]{mcds}. Moreover, $(f\odot g)_{0}$ is the push-out product of $f_{0}$ and the isomorphism $g_{0}$ in the category of pointed simplicial sets or compactly generated topological spaces with the smash product. Hence $(f\odot g)_{0}$ is an isomorphism, so $f\odot g$ is a positive stable (trivial) cofibration.
\end{proof}

\begin{cor}\label{quillen2}
If $\C M$ is any of the categories in Examples \ref{ex} and \ref{ex2}, there are pseudo-cofibrations in $\C M$ which are not cofibrations in $\tilde{\C M}$.
\end{cor}

This follows from the fact that the model structure $\tilde{\C M}$ does not coincide with the ordinary stable model structure.

It would be interesting to know whether $\C M$ has, in general, a model structure with pseudo-cofibrations as cofibrations and the same weak equivalences. That would be a different way, maximal in some sense, of endowing $\C M$ with a model structure with the same weak equivalences and cofibrant tensor unit. It is unclear whether the methods of cofibrantly generated or combinatorial model categories might be useful to answer this question.

We conclude this paper with some applications to stable homotopy theory. The homotopy theory of small categories enriched in symmetric spectra of simplicial sets has been considered in \cite[Theorem 1.10]{htsc}. A recent result of Berger and Moerdijk \cite[Theorem 1.10]{htec} studies the homotopy theory of small categories enriched in a general $\C M$ under some assumptions, including cofibrancy of the tensor unit. In particular their theorem does not apply to the category $\C M_S$ of $S$-modules. 
This hypothesis is not required in \cite{dkhtec}, but combinatoriality is demanded, so model categories of topological nature, like $S$-modules, do not fit either. Nevertheless, Berger--Moerdijk's theorem does apply to $\tilde{\C M}_S$, as we will now see. We start by checking that $\C M_S$ is a symmetric monoidal model category (see Example \ref{ex3} for the very strong unit axiom).

\begin{prop}\label{sar}
The category of $S$-modules $\C M_S$ is cofibrantly generated and satisfies the push-out product axiom and the monoid axiom.
\end{prop}

\begin{proof}
Two sets of generating cofibrations and generating trivial cofibrations in $\C M_S$ are $I=\{f_{q,n}\colon S\wedge_{\C L}\mathbb L\Sigma^\infty_q S^n\r S\wedge_{\C L}\mathbb L\Sigma^\infty_q CS^n\}_{q,n\geq 0}$ and $J=\{g_{q,n}\colon S\wedge_{\C L}\mathbb L\Sigma^\infty_q (CS^n\wedge \{0\}_+)\r S\wedge_{\C L}\mathbb L\Sigma^\infty_q (CS^n\wedge [0,1]_+)\}_{q,n\geq 0}$, respectively, see \cite[Theorem VII.4.14 and the proof of Lemma VII.5.6]{ekmm}. We use the criteria in \cite[Lemma 3.5]{ammmc} to check the two axioms.  The push-out product of two generating cofibrations $f_{p,m}\odot f_{q,n}$ is the $S$-module map obtained by applying the functor $S\wedge_{\C L}\mathbb L\Sigma^\infty_{p+q}-$ to the following map of spaces
$$(S^m\wedge CS^n)\cup_{S^m\wedge S^n}(CS^m\wedge S^n)
\To 
CS^m\wedge CS^n,
$$
see \cite[Proposition II.3.6]{esht} and \cite[Propositions I.6.1, I.8.2 and Definition II.1.1]{ekmm}. This map of spaces is the inclusion of a subcomplex in a CW-complex, hence $f_{p,m}\odot f_{q,n}$ is a cofibration in $\C M_S$. The push-out product of a generating cofibration and a generating trivial cofibration $f_{p,m}\odot g_{q,n}$ is obtained in the same way from the map,
$$(S^m\wedge CS^n\wedge[0,1]_+)\cup_{S^m\wedge S^n\wedge\{0\}_+}(CS^m\wedge S^n\wedge\{0\}_+)
\To 
CS^m\wedge CS^n\wedge[0,1]_+.
$$
This map is the inclusion of a subcomplex which is a deformation retract in a CW-complex, hence $f_{p,m}\odot g_{q,n}$ is a trivial cofibration in $\C M_S$. This proves the push-out product axiom.

The following proof of the monoid axiom is due to Mandell \cite{mandell}. Notice that any map in $$(\ob\C M_S)\wedge_S J=\{X\wedge_{\C L}\mathbb L\Sigma^\infty_q (CS^n\wedge \{0\}_+)\r X\wedge_{\C L}\mathbb L\Sigma^\infty_q (CS^n\wedge [0,1]_+)\}_{\substack{q,n\geq 0\\X\in\ob\C M_S}}$$ is the inclusion of a strong deformation retract. This property is preserved under push-outs. Therefore, it is enough to notice that the transfinite composition in $\operatorname{Top}_*$ of closed inclusions which are also weak equivalences is a weak equivalence \cite[Lemma 2.4.8]{hmc}. 
% Let $\lambda\r\operatorname{Top}\colon\alpha\mapsto X_\alpha$ be a continuous sequence of closed inclusions between compactly generated topological spaces indexed by an ordinal $\lambda$. By \cite[Proposition 2.4.2]{hmc},
% $\pi_0(\colim_{\alpha<\lambda}X_\alpha)=\colim_{\alpha<\lambda}\pi_0(X_\alpha)$ and $\pi_n(\colim_{\alpha<\lambda}X_\alpha,x)=\colim_{\alpha<\lambda}\pi_n(X_\alpha,x)$ for any $n>0$ and $x\in X_0$. Hence, if in addition all  successor bonding maps $X_\alpha\r X_{\alpha+1}$, $\alpha+1<\lambda$, are weak equivalences, we derive that the transfinite composition $X_0\r\colim_{\alpha<\lambda}X_\alpha$ is also a weak equivalence.
\end{proof}

We continue by checking the technical hypotheses of Berger--Moerdijk's theorem.

\begin{lem}
Both $\C M_S$ and $\tilde{\C M}_S$ are compactly generated in the sense of \cite[Definition 1.2]{htec}.
\end{lem}

\begin{proof}
It is enough to check that any object (resp.~any source of a map in $I$) is small (resp.~finite) relative to $\wedge_S$-cofibrations in  $\tilde{\C M}_S$ (where $\otimes=\wedge_S$), see \cite[Definition 1.2 and the paragraph preceding Lemma 1.3]{htec}. Our argument is based in the notion of $h$-cofibration recalled in Example \ref{ex2}, which also makes sense for $S$-modules. If $X$ is any $S$-module and $f$ is an $h$-cofibration, then $X\wedge_S f$ is also an $h$ cofibration since the homotopy extension property for $X\wedge_S f$ with respect to $Y$ is equivalent to the homotopy extension property for $f$ with respect to the internal morphism object $\hom_{\C M_S}(X,Y)$. With the choice in the proof of Proposition \ref{sar}, all maps in $I$ are $h$-cofibrations, since they are obtained by applying $S\wedge_{\C L}\mathbb L\Sigma^\infty_q $ to $h$-cofibrations in $\operatorname{Top}_*$. The map $\varnothing\r\unit$ is an $h$-cofibration for obvious reasons. Therefore any $\wedge_S$-cofibration in $\tilde{\C M}
_S$ is and $h$-cofibration, and in particular a spacewise closed inclusion in $\operatorname{Top}_*$. All spaces in $\operatorname{Top}_*$ are small relative to closed inclusions, and compact spaces are even 
finite. Hence, all objects in $\C M_S$ are small relative to $\wedge_S$-cofibrations and, moreover, the sources of $I$ are finite since they are obtained by applying $S\wedge_{\C L}\mathbb L\Sigma^\infty_q$ to compact spaces (spheres).
\end{proof}

The two previous results ensure the existence of two model structures on the category $\cats{\C M_S}{C}$ of small categories enriched in $S$-modules with a fixed set of objects $C$, see \cite[Remark 3.2 and Proposition 3.3 (2)]{amseco}. Weak equivalences and fibrations are defined locally \cite[Definition 1.6]{htec}, either in $\C M_S$ or in $\tilde{\C M}_S$. These two model structures with the same weak equivalences will be denoted by $\cats{\C M_S}{C}$ and $\cats{\tilde{\C M}_S}{C}$, respectively. We now check the existence of generating sets of intervals in the sense of \cite[Definition 1.11]{htec}.

\begin{lem}
There exist generating sets of $\C M_S$-intervals and $\tilde{\C M}_S$-intervals.
\end{lem}

\begin{proof}
A single $\C M_S$-interval $\mathbb G$ generates since all objects in $\C M_S$ are fibrant. Moreover, the retraction in \cite[Definition 1.11]{htec} can be taken to be a weak equivalence, see \cite[Lemma 2.1]{htec} and its proof. An $\C M_S$-interval is also an $\tilde{\C M}_S$-interval since $\tilde{\C M}_S$ has more cofibrations than $\C M_S$. Let us chech that $\{\mathbb G\}$ is also a generating set of $\tilde{\C M}_S$-intervals.

Any $\tilde{\C M}_S$-interval $\mathbb H$ has a cofibrant resolution $\tilde{\mathbb H}\st{\sim}\r\mathbb H$ in $\cats{\C M_S}{\{0,1\}}$. This $\tilde{\mathbb H}$ is an $\C M_S$-interval by the very definition, so there exists a weak equivalence $\mathbb G\st{\sim}\r\tilde{\mathbb H}$. We factor the composite $\mathbb G\st{\sim}\r\tilde{\mathbb H}\st{\sim}\r\mathbb H$ into a trivial cofibration followed by a trivial fibration in $\cats{\tilde{\C M}_S}{\{0,1\}}$, $\mathbb G\st{\sim}\into \mathbb K\st{\sim}\onto\mathbb H$. The trivial fibration is a retraction since $\tilde{\C M}_S$-intervals are cofibrant in $\cats{\tilde{\C M}}{\{0,1\}}$.
\end{proof}

The following result is a consequence of \cite[Theorems 1.10 and 2.5]{htec}, whose hypotheses have been checked above.

\begin{prop}
The category $\cats{\C M_S}{}$ of all small categories enriched in $S$-modules has a cofibrantly generated right proper model structure where weak equivalences are Dwyer--Kan equivalences \cite[Definition 2.17]{htec} and trivial fibrations are enriched functors surjective on objects which are local surjective trivial fibrations in $\C M_S$.
\end{prop}

This result is also valid if we replace $\C M_S$ with any of the categories $\C M$ in Examples \ref{ex} and \ref{ex2}, however it is less interesting since Berger--Moerdijk's theorem applies directly to the ordinary stable model structures. All these model structures on enriched categories are Quillen equivalent. This can be shown by using the strong symmetric monoidal Quillen equivalences $\operatorname{Sp}^\Sigma\rightleftarrows\Sigma \C S\rightleftarrows\C I\C S\rightleftarrows\C W\C T$ and $\Sigma \C S\rightleftarrows\C M_S$ in \cite{mcds, smss}, see \cite[Corollary 1.14 and the paragraph afterwards]{htec} and Corollary \ref{quillen2}.

Caviglia's \cite[Proposition 3.3 (2)]{amseco} also implies the existence of model structures on nonsymmetric  coloured operads and reduced symmetric coloured operads with a fixed ser of colours enriched in $\C M_S$ or  $\tilde{\C M}_S$. Reduced means that there are no non-trivial arity zero operations. Weak equivalences and fibrations are defined locally, see \cite[Definition 4.5]{amseco}. The following result is a consequence of \cite[Lemma 4.8, Theorem 4.22 and Propositions 4.25 and 5.4]{amseco}. The hypotheses have been checked above.

\begin{prop}
The category of all nonsymmetric  coloured operads and the category of all reduced symmetric coloured operads enriched in $S$-modules have a cofibrantly generated right proper model structure where weak equivalences are Dwyer--Kan equivalences \cite[Definition 4.24]{amseco} and trivial fibrations are local surjective trivial fibrations in $\C M_S$ which are surjective on colours.
\end{prop}

Again, this result is valid but maybe not very relevant for the categories of Examples \ref{ex} and \ref{ex2}, since the aforementioned Caviglia's results apply to the ordinary stable model structures.

\subsection*{Acknowledgements}

The author is very grateful to Andrew Blumberg, Javier J.~Guti\'errez, Mark Hovey, Michael A.~Mandell, Peter May, Stefan Schwede, Brooke Shipley, and David White for illuminating conversations related to the contents of this paper. He was partially supported
by the Andalusian Ministry of Economy, Innovation and Science under the grant FQM-5713 and by the Spanish Ministry of Economy and Competitiveness under the MEC-FEDER grant  MTM2013-42178.

% ----------------------------------------------------------------
% \bibliographystyle{amsalpha}
% \bibliography{../Fernando}

\begin{thebibliography}{LMSM86}

\bibitem[BM13]{htec}
C.~Berger and I.~Moerdijk, \emph{{On the homotopy theory of enriched
  categories}}, Q. J. Math. \textbf{64} (2013), no.~3, 805--846.

\bibitem[Cav14]{amseco}
G.~Caviglia, \emph{{A model structure for enriched coloured operads}},
  \texttt{arXiv:1401.6983 [math.AT]}, January 2014.

\bibitem[EKMM97]{ekmm}
A.~D. Elmendorf, I.~Kriz, M.~A. Mandell, and J.~P. May, \emph{{Rings, modules,
  and algebras in stable homotopy theory}}, {Mathematical Surveys and
  Monographs}, vol.~47, American Mathematical Society, Providence, RI, 1997,
  With an appendix by M. Cole.

\bibitem[Hov99]{hmc}
M.~Hovey, \emph{{Model categories}}, {Mathematical Surveys and Monographs},
  vol.~63, American Mathematical Society, Providence, RI, 1999.

\bibitem[HSS00]{se}
M.~Hovey, B.~Shipley, and J.~Smith, \emph{{Symmetric spectra}}, J. Amer. Math.
  Soc. \textbf{13} (2000), no.~1, 149--208.

\bibitem[LM07]{mmmc}
L.~G. {Lewis Jr.} and M.~A. Mandell, \emph{{Modules in monoidal model
  categories}}, J. Pure Appl. Algebra \textbf{210} (2007), no.~2, 395--421.
  
\bibitem[LMSM86]{esht}
L.~G. {Lewis Jr.}, J.~P. May, M.~Steinberger, and J.~E. McClure,
  \emph{{Equivariant stable homotopy theory}}, {Lecture Notes in Mathematics},
  vol. 1213, Springer-Verlag, Berlin, 1986, With contributions by J. E.
  McClure.

\bibitem[Lur14]{lurieha}
J.~Lurie, \emph{{Higher {A}lgebra}}, available at the author's web page:
  \texttt{http://www.math.harvard.edu/\~{}lurie}, September 2014.
  
\bibitem[Man14]{mandell}
M.~A.~Mandell, private email communication, 7 November 2014.

\bibitem[MMSS01]{mcds}
M.~A.~Mandell, J.~P. May, S.~Schwede, and B.~Shipley, \emph{{Model categories
  of diagram spectra}}, Proc. London Math. Soc. (3) \textbf{82} (2001), no.~2,
  441--512.

\bibitem[Mur14a]{dkhtec}
F.~Muro, \emph{{Dwyer-Kan homotopy theory of enriched categories}},
  \texttt{arXiv:1201.1575 [math.AT]}, March 2014.

\bibitem[Mur14b]{htnso2}
\bysame, \emph{{Homotopy theory of non-symmetric operads, {II}: {C}hange of
  base category and left properness}}, Algebr. Geom. Topol. \textbf{14} (2014),
  229--281.

\bibitem[Mur14c]{manso}
\bysame, \emph{{Moduli spaces of algebras over nonsymmetric operads}}, Algebr.
  Geom. Topol. \textbf{14} (2014), no.~3, 1489--1539.
  
 \bibitem[Sch01]{smss}
S.~Schwede, \emph{{$S$}-modules and symmetric spectra}, Math. Ann. \textbf{319}
  (2001), no.~3, 517--532.

\bibitem[Shi04]{acmccrs}
B.~Shipley, \emph{{A convenient model category for commutative ring spectra}},
  {Homotopy theory: relations with algebraic geometry, group cohomology, and
  algebraic {$K$}-theory}, {Contemp. Math.}, vol. 346, Amer. Math. Soc.,
  Providence, RI, 2004, pp.~473--483.

\bibitem[SS00]{ammmc}
S.~Schwede and B.~Shipley, \emph{{Algebras and modules in monoidal model
  categories}}, Proc. London Math. Soc. (3) \textbf{80} (2000), no.~2,
  491--511.

\bibitem[SS03]{emmc}
\bysame, \emph{{Equivalences of monoidal model categories}}, Algebr. Geom.
  Topol. \textbf{3} (2003), 287--334 (electronic).
  
\bibitem[Tab09]{htsc}
G.~Tabuada, \emph{{Homotopy theory of spectral categories}}, Adv. Math.
  \textbf{221} (2009), no.~4, 1122--1143.

\bibitem[TV08]{hagII}
B.~To{\"e}n and G.~Vezzosi, \emph{{Homotopical algebraic geometry. {II}.
  {G}eometric stacks and applications}}, Mem. Amer. Math. Soc. \textbf{193}
  (2008), no.~902, x+224.
  
\bibitem[Whi12]{whimo}
D.~White, \emph{Hovey's unit axiom in monoidal model categories}, URL (version: 2012-01-21): \texttt{http://mathoverflow.net/q/85995}.

\end{thebibliography}

\providecommand{\bysame}{\leavevmode\hbox to3em{\hrulefill}\thinspace}
\providecommand{\MR}{\relax\ifhmode\unskip\space\fi MR }
\providecommand{\MRhref}[2]{%
  \href{http://www.ams.org/mathscinet-getitem?mr=#1}{#2}
}
\providecommand{\href}[2]{#2}

\end{document}